\documentclass[11pt]{article} 
\usepackage[T1]{fontenc} 
\usepackage[latin1]{inputenc}
\usepackage{lmodern} 
\usepackage[leqno]{amsmath}
\usepackage{amssymb,amscd,amsthm,amsmath,mathrsfs,yfonts,manfnt,pb-diagram,pb-xy}
\usepackage{bbm}

\usepackage{geometry} 
\usepackage{graphicx} 

\usepackage{booktabs} 
\usepackage{array} 
\usepackage{paralist} 
\usepackage{verbatim} 
\usepackage{subfig} 

\usepackage{fancyhdr} 
\usepackage[nottoc,notlof,notlot]{tocbibind} 
\usepackage[titles,subfigure]{tocloft} 
\usepackage{hyperref}

\newcommand{\Aa}{\mathscr{A}}
\newcommand{\AlgLie}{\mathfrak{g}}
\newcommand{\C}{\mathbbm{C}}
\newcommand{\Hh}{\mathscr{H}}
\newcommand{\Ll}{\mathcal{L}}

\newcommand{\mc}{\mathcal}

\newcommand{\mT}{\ensuremath{\mathbb{T}}}
\newcommand{\N}{\mathbbm{N}}
\newcommand{\Oo}{\mathcal{O}}
\newcommand{\R}{\mathbbm{R}}

\newcommand{\Z}{\mathbbm{Z}}

\DeclareMathOperator{\dom}{Dom}

\DeclareMathOperator{\id}{id}
\DeclareMathOperator{\GNS}{GNS}

\DeclareMathOperator{\Tr}{Tr}

\DeclareMathOperator{\Ran}{Ran}

\theoremstyle{definition}		
\newtheorem{Def}{Definition}[section]
\newtheorem{Not}[Def]{Notation}
\newtheorem{Ex}[Def]{Example}

\theoremstyle{plain}
\newtheorem{Prop}[Def]{Proposition}
\newtheorem{Th}[Def]{Theorem}
\newtheorem{Lem}[Def]{Lemma}

\theoremstyle{remark}
\newtheorem{Rem}[Def]{Remark}

\theoremstyle{definition}		
\newtheorem{Assumption}[Def]{Assumption}

\title{Ergodic Actions and Spectral Triples}
\author{Olivier \textsc{Gabriel} \& Martin \textsc{Grensing}}

\begin{document}

\maketitle

\abstract{
In this article, we give a general construction of spectral triples from certain Lie group actions on unital $C^*$-algebras. If the group $G$ is compact and the action is ergodic, we actually obtain a real and finitely summable spectral triple which satisfies the first order condition of Connes' axioms. This provides a link between the ``algebraic'' existence of ergodic action and the ``analytic'' finite summability property of the unbounded selfadjoint operator. More generally, for compact $G$ we carefully establish that our (symmetric) unbounded operator is essentially selfadjoint. Our results are illustrated by a host of examples -- including noncommutative tori and quantum Heisenberg manifolds.
}

\medbreak




%

\section{Introduction}

The Gelfand-Naimark theorem establishes an equivalence of categories between locally compact topological spaces and commutative $C^*$-algebras, which one takes advantage of to define noncommutative spaces. 

In 1980, Connes introduced in \cite{CAlgGeoDiff} what he called a ``differential structure'' induced by a Lie group action on a $C^*$-algebra. This early notion was later superseded by the framework of \emph{spectral triples}, devised axiomatically by Connes in \cite{Reality,Grav}.

A choice of spectral triple comes down to fixing an unbounded operator on a representation space for a $C^*$-algebra which corresponds to a Dirac-type operator and is thus (in the unital case) supposed to be a self-adjoint operator of compact resolvent (see Def. \ref{Def:TrSp}).

Lie group actions and spectral triples thus provide two different approaches to ``smoothness'' for noncommutative spaces -- for example, the boundedness of the commutators $[D,a]$ is a measure of regularity for $a$. The validity of this approach was confirmed to some extent by Connes' \emph{reconstruction theorem} \cite{recons}: under slight strengthening of the axioms of \cite{Grav}, spectral triples on commutative $C^*$-algebras arise from smooth manifolds.

\medbreak

If we think of $C^*$-algebras as sets of continuous functions on NC spaces, to study a ``smooth noncommutative manifold'' requires an analog of smooth functions on this ``manifold''. In other words, we need a ``{smooth subalgebra}'' $\Aa \subseteq A$. Such smooth subalgebras can be obtained in several ways (see \textsl{e.g.} \cite{BlackCuntz, OkaKThBost}). Here, we follow the familiar construction (see Prop. \ref{Prop:Aa} below) obtaining $\Aa$ from an action of Lie group $G$ on $A$. A natural question arises from this construction: if we consider a sub-group $G_0 \subseteq G$, how do we distinguish the two associated smooth subalgebras $\Aa_0$ and $\Aa$? 

In our case, this difficulty is solved by ergodicity: if the action of $G$ is ergodic, we can construct a summable spectral triple for $\Aa$ -- which need not be possible for $G_0$, if its action is not ergodic (see Example \ref{Ex:Tore1NC} below).

%
%
%
%

\bigbreak

Our paper originates from Rieffel's article \cite{Rieffel98} in which he considers unbounded Dirac-type operators of the form
\begin{equation}
\label{Eqn:DiracIni}
D = \sum \; \partial_j\otimes F_j,
\end{equation}
 where $\partial_j$ denotes the infinitesimal generators of a Lie group action and $F_j$ are generators of a Clifford algebra (see \cite{Rieffel98}, p.226). His paper also puts special emphasis on ergodic actions. 
The expression \eqref{Eqn:DiracIni} appeared in \cite{GeomQHM} too, where it is presented as a ``general principle of construction of spectral triple''.

In the present document we determine conditions under which Rieffel's operator yields a spectral triple in the sense of Connes (see Def. \ref{Def:TrSp} below) and study its summability properties.  As \cite{Rieffel98} seems to suggest, we obtain a finitely {summable} spectral triple for ergodic actions of compact Lie groups. More precisely (see Theorem \ref{Thm:Main} below for the exact forms of the spectral triples):
\begin{Th}
If $G$ is a compact Lie group of dimension $n$ acting ergodically on a unital $C^*$-algebra $A$, then
\begin{enumerate}[1.]
\item
there is an explicit $n^+$-summable spectral triple on $A$, which is even when $n$ is even, has a real structure and satisfies the first order condition;
\item 
 if we are given a covariant representation of $(A,G)$ on $\Hh_0$ which satisfies a certain finiteness condition, then we can manufacture a $n^+$-summable spectral triple from it. 
\end{enumerate}
\end{Th}
The above theorem is relevant on two counts:
\begin{itemize}
\item
Point 1. recovers for instance the usual spectral triples for noncommutative tori of any dimension (see Section \ref{Sec:Example});
\item 
Point 2. links algebraic or ``geometric'' properties -- namely the existence of a covariant representation -- with analytic properties \textsl{i.e.} the selfadjointness and finite summability of $D$.
\end{itemize}
For covariant representations of non-compact groups of dimension $n$, we obtain a symmetric operator with bounded commutators (see Prop. \ref{Prop:BasicDirac}), which is graded when $n$ is even. Furthermore, if the Hilbert space of the triple comes from a $G$-invariant trace \textsl{via} the GNS construction, an associated real structure is available (see Prop. \ref{Prop:BasicDirac2}), thereby refining the ``general principle'' mentioned in \cite{GeomQHM}.

\medbreak

A general mean of obtaining spectral triples is given in \cite{DeformationsConnesLandi}. This construction is similar to ours in the sense that it assumes a certain symmetry on the initial space -- a Riemannian manifold whose isometry group has rank at least $2$ in \cite{DeformationsConnesLandi}, an ergodic action of compact Lie group for us -- and estimates the summability of the resulting spectral triple. On the one hand, they rely on a deformation of a (commutative) geometric situation, while we have purely ``noncommutative'' assumptions, but on the other hand, their result yields orientability and Poincaré duality besides the summability, real structure and first order properties (see \cite{Grav} for the definitions of these axioms).

Our results more closely ressemble the general construction presented in \cite{TrSpLieGpWahl}. Nevertheless, the aforementioned article focuses on semifinite spectral triples and their index properties in the setting of general actions of compact Lie groups, while we put emphasis on ``regular'' spectral triples and their summability properties in the case of ergodic actions.

The notion of ``ergodic action'' -- which plays a crucial role in our results -- is well-studied, but our argument depends only on the seminal work \cite{ErgodCpctGpHKLS} of H\o{}egh-Krohn, Landstad and St\o{}rmer. We expect the vast literature on this topic to provide us eventually with new tractable classes of examples. However, the many classical examples of spectral triples are already almost entirely covered by our framework, as we show in the last section.

Another point of view on our results is that this article (together with the forthcoming \cite{TrSpPCGGG}) provides some sort of ``backward compatibility'' of the original article \cite{CAlgGeoDiff} with the more recent framework of spectral triples.

\bigbreak

This article starts with a preliminary section \ref{Sec:DefTrSp}, defining precisely the notion of spectral triple that we will use. We prove that, given a covariant representation of $A$ and $G$ on a Hilbert space $\Hh_0$, a symmetric unbounded operator $D$ with bounded commutators arises naturally. In the next section \ref{Sec:GNS}, we proceed with the particular case when $\Hh_0$ arises from the GNS construction and show that a real structure (implying the existence of a selfadjoint extension of $D$) exists in this case. Going back to general covariant representations, we establish carefully in section \ref{Sec:Compact} that if $G$ is compact, $D$ is essentially selfadjoint. These two threads of results are finally combined in Section \ref{Sec:Main}, where the main theorem is established. Finally, the last section \ref{Sec:Example} relates our results to prior work, by examining remarks, examples and counterexamples.

\section{Spectral triples and covariant representations}
\label{Sec:DefTrSp}

For most of this article, we only consider \emph{unital} $C^*$-algebras and \emph{nondegenerated representations} $\pi \colon A \to B(\Hh)$ of $C^*$-algebras, \textsl{i.e.} $\pi(1_A) = \id_\Hh$. In Section \ref{Sec:Example}, some examples involve nonunital $C^*$-algebras, in which case condition (i) of Def. \ref{Def:TrSp} is replaced by ``$\pi(a)(1+D^2)^{-1}$ is compact for all $a\in A$''.

\smallbreak

The expression ``spectral triple'' has been used to denote several different notions -- thus we need here to carefully define the meaning we give to it. A basic definition of this expression is the following:
\begin{Def} 
\label{Def:TrSp} Let $A$ be a $C^*$-algebra. An odd \emph{spectral triple}, also called odd \emph{unbounded Fredholm module}, is a triple $(\pi, \Hh, D)$ where:
\begin{itemize}
\item
$\Hh$ is a Hilbert space and $\pi \colon A\to B(\Hh)$ a $*$-representation of $A$ as bounded operators on $\Hh$
\item
a selfadjoint unbounded operator $D$ -- which we will call the \emph{Dirac operator} -- defined on the domain $\dom(D)$
\end{itemize}
such that
\begin{enumerate}[(i)]
\item
$(1+D^2)^{-1/2}$ is compact,
\item
the subalgebra $\Aa$ of all $a\in A$ such that
\begin{align*}
\pi(a) (\dom(D)) &\subseteq \dom(D)
&
[D,\pi(a)] & \text{ extends to a bounded map on }\Hh
\end{align*}
is dense in $A$.
\end{enumerate}
An \emph{even spectral triple} is given by the same data, but we further require that a grading $\gamma$ be given on $\Hh$ such that (i) $A$ acts by even operators, (ii) $D$ is odd.
\end{Def}

\begin{Rem}
In the above definition, we do not require that the representation $\pi$ be faithful. However, many references (including Connes' articles) on this topic include this additional constraint.
\end{Rem}

The situation is further complicated by the different additional properties which one can require on spectral triples. The most complete collection of such ``axioms'' is surely the one proposed by Connes in his article \cite{Grav} (later amended in \cite{recons} to prove his reconstruction theorem). Here, we will only need a few of these. In (essential) accordance with the nomenclature of \cite{Grav} and \cite{recons}, we call them \emph{reality}, \emph{order one}, \emph{finite summability} and \emph{finiteness}. References for them will be given when needed.


%

\begin{Assumption}
\label{Assumption1}
In this document, we assume that $G$ is a Lie group of finite dimension $n$ which acts on a $C^*$-algebra $A$ \textsl{via} a strongly continuous action $\alpha$. We denote by $\AlgLie $ the Lie algebra of $G$.
\end{Assumption}

Such a $G$-action defines a ``smooth version'' $\Aa$ of $A$ (see for instance Prop. 3.45 p.138 of \cite{EltNCG}): 
\begin{Prop}
\label{Prop:Aa}
The subalgebra $\Aa$ of $G$-smooth elements:
$$ \Aa := \big\{ a \in A : g \mapsto \alpha_g(a) \text{ is in }C^\infty (G, A) \big\} $$
is a dense sub-$*$-algebra of $A$, with a natural Fréchet structure, which is stable under holomorphic functional calculus.
\end{Prop}

\begin{Def}
\label{Def:Covar}
Under Assumption \ref{Assumption1}, a \emph{covariant representation of $A$ and $G$ on a Hilbert space $\Hh_0$} is a representation $\pi$ of $A$ together with a unitary and strongly continuous representation $U$ of $G$ on $\Hh_0$ which satisfy the following compatibility condition:
\begin{equation}
\label{Eqn:Cov}
\pi(\alpha_g(a)) \xi = U_g \pi(a) U_g^* \xi ,
\end{equation}
for all $a \in A$, $g \in G$ and $\xi \in \Hh_0$.
\end{Def}
In the present article, we sometimes abuse notations and write $a \xi$ instead of $\pi(a) \xi$. 

Given a covariant representation as above, we define a smooth domain $\Hh_0^\infty \subseteq \Hh_0$, using the same process as in Prop. \ref{Prop:Aa} \textsl{i.e.} 
\begin{equation}
\label{Def:Hh0inf}
\Hh_0^\infty  := \big\{ \xi \in \Hh_0 : g \mapsto U_g \xi \text{ is in }C^\infty (G, \Hh_0) \big\}
\end{equation}
It is clear from the definitions that $\Aa \Hh_0^\infty \subseteq \Hh_0^\infty $. Moreover, $\Hh_0^\infty$ is dense in $\Hh_0$, since it contains the dense ``G\r{a}rding's domain'' described in \cite{ReedSimonI}, p.306 
and initially introduced in \cite{LieGroupsGarding}.

\begin{Not}
\label{Not:Base}
We fix a basis $(X_j)_{j=1}^n$ of $\AlgLie $ and denote by
\begin{itemize}
\item
$\partial _j^\Aa$ the associated infinitesimal generators of the action of $G$ on $A$ defined by
$$\partial_j^\Aa(a):=\lim_{t\to 0} \frac{\alpha_{\exp(t X_j)}(a)-a}{t}$$
for any $a \in \Aa$;
\item
$\partial _i$ the associated generators of the action of $G$ on $\Hh_0$ defined in the same way. Since $U$ is unitary, which satisfy
\begin{equation}
\label{Eqn:SymGenInf}
 \langle \partial _j \xi, \eta \rangle + \langle \xi, \partial _j \eta \rangle = 0.
\end{equation}
\end{itemize}
\end{Not}

For all $\xi \in \Hh_0^\infty $ and all $a \in \Aa$, taking the derivative of \eqref{Eqn:Cov} we clearly have:
$$ \partial _j^\Aa(a) \xi = \partial_j(a \xi) - a \partial_j \xi .$$
In other words, for all $a \in \Aa$
\begin{equation}
\label{Eqn:Comm}
 [\partial _j, a] = \partial^\Aa _j(a) .
\end{equation}
%

\begin{Def}
Following \cite{EltNCG} p.174, for $n \in \N$ we denote $\C l(n)$ the universal unital $C^*$-algebra generated by $n$ selfadjoint elements $e_i$ which satisfy the relations:
\begin{equation}
\label{Eqn:Spin}
 e_j e_k + e_k e_j = 2 \delta_{jk} .
\end{equation}
A $\Z/2 \Z$-grading on $\C l(n)$ is induced by the automorphism $h$ defined by $h(e_i) = - e_i$. We denote $\C l^+(n)$ and $\C l^-(n)$ the $+1$ and $-1$ eigenspaces of $h$.
\end{Def}


The following identifications appear as Lemma 5.5 p.178 of \cite{EltNCG}: if $N = 2^m$, then
\begin{align}
\label{Eqn:IddCliff}
\C l(2 m) &= M_N(\C)
&
\C l(2m +1) &= M_N(\C) \oplus M_N(\C).
\end{align}
It follows from the above identification that, up to unitary equivalence, there is a unique representation of $\C l(n)$ for even $n$, and that there are two inequivalent representations for odd $n$. We fix $m$ such that $n =2 m$ (even case) or $n = 2 m +1$ (odd case) and define a \emph{chirality element} $\gamma$ (compare \cite{EltNCG}, Def. p.179) by
$$ \gamma = (-i)^m e_1 \cdots e_n .$$
$\gamma$ induces a grading operator, namely it satisfies $\gamma^* = \gamma$ and $\gamma^2 = 1$. Moreover $\gamma e_j \gamma = -e_j$ (in the even case) or $\gamma e_j \gamma = e_j$ (in the odd case). In other words, the grading in this case is inner. Hence, in the odd case, $\gamma$ is in the center of $\C l(n)$. For irreducible representations, $\gamma$ has to be sent to either $1$ or $-1$ -- and this distinguishes between the two possible irreducible representations of $\C l(n)$ for odd $n$. This also justifies that the chirality element does not appear in irreducible representations of odd Clifford algebras.


\bigbreak

Section 2 of the article \cite{ProductTrSpDD} provides an explicit set of generators of the irreducible representations of $\C l(n)$ for all $n$, together with an explicit involution $J_S$ and (in the even case) a grading operator $\gamma_S$. We summarise these results in the following:
\begin{Prop}
\label{Prop:MatF}
Consider a positive integer $n$ and an irreducible representation of $\C l(n)$ on a vector space $S$. Up to unitary equivalence, it is determined by $n$ matrices $F_j$ s.t.
\begin{align}
\label{Eqn:RelF}
F_j^* &= - F_j 
&
F_j F_k + F_k F_j &= - 2 \delta_{jk}.
\end{align}
If $n$ is even, a grading operator $\gamma_S$  is available which satisfies $\gamma_S^* = \gamma_S$, $\gamma_S^2 = 1$ and $\gamma_S F_j = -F_j \gamma_S$ for all $j$. There is an explicit anti-linear map $J_S$ s.t. for all $j = 1, \ldots , n$ and $s, s' \in S$,
\begin{align*}
\langle J_S s, J_S s' \rangle &= \langle s', s \rangle 
&
J_S^2 &= \varepsilon_J
&
J_S F_j &= \varepsilon_D F_j J_S
&
J_S \, \gamma_S &= \varepsilon_\gamma \gamma_S\,  J_S,
\end{align*}
where 
\begin{itemize}
\item
the last equality (and therefore $\varepsilon_\gamma$) only appears in the even cases,
\item
$\varepsilon_J, \varepsilon_D$ and $\varepsilon_\gamma$ are all $-1$ or $1$, the proper sign depending on $n$ modulo $8$:
\begin{center}
\begin{tabular}{cccccccccc}
\toprule
$n$ & $0$ & $2$ & $4$ & $6$ && $1$ & $3$ & $5$ & $7$ \\
\midrule
$\varepsilon_J$ & $+$ & $-$ & $-$ & $+$ && $+$ & $-$ & $-$ & $+$\\
$\varepsilon_D$ & $+$ & $+$ & $+$ & $+$ && $-$ & $+$ & $-$ & $+$\\
$\varepsilon_\gamma$ & $+$ & $-$ & $+$ & $-$ && & & & \\
\bottomrule
\end{tabular}
\end{center}
\end{itemize}
\end{Prop}

\begin{Rem}
In the original article, the matrices $F_j$ associated to $\C l(n)$ are denoted $\gamma^{j}_{(n)}$ (for $n$ even) and $\gamma^{j}_{(n), \pm}$ (for $n$ odd, the sign corresponding to the two irreducible representations) in \cite{ProductTrSpDD}. If the representation of $\C l(n)$ is denoted $\pi$, these matrices $F_j$ correspond to $i \pi(e_j)$ for $e_j$, generators of $\C l(n)$. In the even case, the article \cite{ProductTrSpDD} actually isolate two possible antilinear maps denoted $J_\pm$ (see Section 2.3 p.1836).
\end{Rem}

We are now ready to construct a symmetric unbounded operator $D$:
\begin{Prop}
\label{Prop:BasicDirac}
Under Assumption \ref{Assumption1}, and for the representation of $\C l(n)$ of Prop. \ref{Prop:MatF}, the equation:
\begin{equation}
\label{Eqn:Dirac}
 D = \sum_{j=1}^n \partial _j \otimes F_j ,
\end{equation}
defines a symmetric unbounded operator $D$ on $\Hh:= \Hh_0 \otimes S$ with domain $\dom(D)=\Hh^\infty _0 \otimes_\C S$. Moreover, 
\begin{enumerate}[(i)]
\item 
 for any $ a \in \Aa$, the commutators with $D$ are bounded. More precisely:
\begin{align*}
a (\dom(D)) &\subseteq \dom(D)
&
[D, a] &= \sum \partial_j^\Aa(a) \otimes F_j;
\end{align*}
\item
if $n$ is even, there is a selfadjoint \emph{grading operator} $\gamma$ such that for all $a \in A$,
\begin{align*}
\gamma^2 &= 1
&
\gamma a &= a \gamma
&
\gamma (\dom D) &\subseteq \dom(D)
&
\gamma D = - D \gamma;
\end{align*} 
\end{enumerate}
\end{Prop}

\begin{Rem}
This Proposition, together with Prop. \ref{Prop:Aa}, already proves the point (ii) of Def. \ref{Def:TrSp}. Point (i) of the same definition will not appear until Theorem \ref{Thm:Main}. 

When $n$ is even, the operator $\gamma$ furnishes provisions for an \emph{even} spectral triple.
\end{Rem}

\begin{proof}
It is clear from the definition of $\Hh_0^\infty $ that $D$ is defined on $\Hh_0^\infty \otimes S$. Let us first prove that $D$ is symmetric on this domain: take $\xi \otimes s$ and $\xi' \otimes s'$ in $\Hh_0^\infty \otimes S$. Relying on \eqref{Eqn:SymGenInf}, we get:
\begin{multline*}
 \langle \xi \otimes s, D( \xi' \otimes s') \rangle = \left\langle \xi \otimes s, \sum \partial _j \xi' \otimes F_j s' \right\rangle \\
= \sum \left\langle \xi, \partial _j \xi' \right\rangle \, \langle s, F_j s' \rangle = \sum \left\langle - \partial _j \xi, \xi' \right\rangle \, \langle -F_j s, s' \rangle \\
= \left\langle \sum \partial _j\xi \otimes E_j s, \xi' \otimes s' \right\rangle = \langle D( \xi \otimes s), \xi' \otimes s' \rangle. 
\end{multline*}
Any $a \in \Aa$ sends $\Hh_0^\infty \otimes S$ to itself and $[D, a]$ extends to a bounded operator: this is obvious from the definitions of $\Aa$ and $\Hh_0^{\infty}$ together with equation \eqref{Eqn:Comm}.

\smallbreak

 To define the grading operator, we use the notations of Prop. \ref{Prop:MatF} and set $\gamma := 1 \otimes \gamma_S$. We clearly get $\gamma^2 = 1$ and $\gamma^* = \gamma$. Since $\gamma$ only acts on $S$ in the tensor product $\Hh = \Hh_0 \otimes S$, while $a \in A$ acts only on $\Hh_0$, they clearly commute. Moreover, $\dom(D)$ is clearly sent to itself by $\gamma$ and the anticommutation relation with $D$ is then easily checked using the properties of $\gamma_S$.
\end{proof}

\section{GNS representation}
\label{Sec:GNS}

\begin{Lem}
\label{Lem:GinvTrace}
Under Assumption \ref{Assumption1}, given a $G$-invariant trace $\tau$ on $A$, the Hilbert space $\Hh_0$ obtained by the GNS construction from $(A, \tau)$ is equipped with a covariant representation in the sense of Def. \ref{Def:Covar}.
\end{Lem}

\begin{proof}
From the definition of $\Hh_0 = \GNS(A, \tau)$, the image $H:= \{ [a], a \in A \}$ of $A$ in $\Hh_0$ is dense. We define the representations of $A$ and $G$ on this subset by:
\begin{align*}
\pi(a)\big( [a'] \big) &= [a a']
&
U_g\big( [a] \big) &= [\alpha_g(a)].
\end{align*}
It is readily checked from these expressions that \eqref{Eqn:Cov} is satisfied. Let us now prove that $U_g$ is unitary:
$$
\langle U_g \big( [a] \big), U_g \big( [a'] \big) \rangle = \tau( \alpha_g(a)^* \alpha_g(a') ) = \tau(a^* a') = \langle [a], [a'] \rangle 
$$
since $\tau$ is $G$-invariant. Since $\alpha$ is strongly continuous on $A$, it is clear that $U$ is strongly continuous on $H$. It then follows from a standard density argument that $U$ is strongly continuous on $\Hh_0$.
\end{proof}

The previous discussion of the real structure is motivated by the following slight generalisation of Theorem X.3 of \cite{ReedSimonII}: 
\begin{Lem}
\label{Lem:UnboundExt}
If $D$ is an unbounded formally selfadjoint operator on a Hilbert space $\Hh$ whose domain is $\dom(D) \subset \Hh$ and $J$ is an antilinear map such that for all $\xi, \eta \in \Hh$,
\begin{equation}
\label{Eqn:NormPres}
\langle J \xi, J \eta \rangle = \langle \eta, \xi \rangle 
\end{equation}
\begin{align*}
J^2 &= \varepsilon_J
&
J(\dom(D)) & \subseteq \dom(D)
&
J D &= \varepsilon_D D J,
\end{align*}
where $\varepsilon_J$ and $\varepsilon_D$ in $\{-1, 1\}$, then $D$ admits a selfadjoint extension.
\end{Lem}

\begin{Rem}
An antilinear operator $J$ which satisfies the equation \eqref{Eqn:NormPres} for all $\xi, \eta \in \Hh$ is called \emph{norm-preserving}.
\end{Rem}

\begin{proof}
We first consider the cases $\varepsilon_D = 1$. If $\varepsilon_J =1$, we are back to the hypotheses of Theorem X.3 p.143 of \cite{ReedSimonII}, thus the conclusion holds.

If $\varepsilon_J = -1$, we reduce the situation to the previous case by an easy computation using tensor products: define an antilinear map $C$ on $\C^2$ by
$
C \begin{pmatrix}
x_1 \\ x_2
\end{pmatrix}
=
\begin{pmatrix}
-\overline{x_2} \\
\overline{x_1}
\end{pmatrix}
$
and then set
\begin{align*}
D_2 &= D \otimes 1
&
J_2 &= J \otimes C,
\end{align*}
where the domain of $D_2$ is $\dom(D) \oplus \dom(D)$. Clearly, the tensor product (over $\C$) of two antilinear maps is well-defined, hence $J_2$ is antilinear. The following facts are easily checked using tensor products: $D_2$ is formally selfadjoint, $J_2^2 =1$, $J_2$ preserves the norm of $\Hh \otimes \C^2$ and the domain of $D_2$. Finally, $J_2$ and $D_2$ commute. Hence, we are back to the previous hypotheses and $D_2$ admits a selfadjoint extension, denoted $\tilde{D_2}$. It is written as a diagonal block matrix, thus the upper left entry denoted $\tilde{D}$ is already a selfadjoint operator, which extends $D$. 

\medbreak

It remains to treat the cases of $\varepsilon_D = -1$. We use the same kind of argument: introduce on $\C^2$ the antilinear operator $
C' \begin{pmatrix}
x_1 \\ x_2
\end{pmatrix}
=
\begin{pmatrix}
\overline{x_2} \\
\overline{x_1}
\end{pmatrix}
$
and then set
\begin{align*}
D_2 &= D \otimes \begin{pmatrix}
1 & 0 \\ 0 & -1
\end{pmatrix}
&
J_2 &= J \otimes C,
\end{align*}
the domain of $D_2$ being $\dom(D) \oplus \dom(D)$. Clearly, we have $J_2 ^2 = J^2 \otimes 1_{\C^2} = \varepsilon_J $, $J_2$ preserves the norm of $\Hh \otimes \C^2$ and sends $\dom(D_2)$ to itself.

Moreover, $J_2$ and $D_2$ commute algebraically, as a simple computation using tensor products shows. 

We are then back to the previous cases of $\varepsilon_D=1$. Hence, $D_2$ admits a selfadjoint extension and we can apply the same argument as in the previous case to extract a selfadjoint extension of $D$.
\end{proof}

In the framework of Lemma \ref{Lem:GinvTrace}, more properties for the Dirac operator are available:
\begin{Prop}
\label{Prop:BasicDirac2}
Under Assumption \ref{Assumption1}, if $\Hh_0 = \GNS(A, \tau)$ for a $G$-invariant trace the symmetric unbounded operator $D$ defined by \eqref{Eqn:Dirac} has the following further properties:
\begin{enumerate}[(i)]
\item
the operator $D$ admits a real structure, \textsl{i.e.} there is a norm-preserving antilinear map $J \colon \Hh \to \Hh$ s.t. for all $a, b \in A$
\begin{align*}
[a, J b^* J^{-1}] &=0
&
J (\dom(D)) &\subseteq \dom(D)
&
J D &= \varepsilon_D D J
\end{align*}
together with $J \, \gamma = \varepsilon_\gamma \gamma \, J$ and $J^2 = \varepsilon_J$ where $\varepsilon_J, \varepsilon_D$ and (in the even case) $\varepsilon_\gamma$ are given by the table of Prop. \ref{Prop:MatF};
\item
$D$ and $J$ satisfy the \emph{first order condition}, \textsl{i.e.} for all $a, b \in \Aa$, 
$$ \big[ [D, a], J b^* J^{-1} \big] = 0 ;$$
\item 
 $D$ admits a selfadjoint extension $\widetilde{D}$.
\end{enumerate}
\end{Prop}

\begin{Rem}
The above Proposition relates to the axioms of \cite{Grav,recons} in the following way:
\begin{itemize}
\item
Point (i) precisely shows that the \emph{reality} condition of \cite{Grav} (see Axiom (7') p.163).
\item
Point (ii) is the first order condition denoted Axiom (2') in p.164 of \cite{Grav}.
\end{itemize}
Of course, we should not forget that at this point, we still do not have a spectral triple in the proper sense, since $(1 + D^2)^{-1/2}$ does not need be compact!
\end{Rem}

\begin{proof}
Using the same notations as in the proof of Lemma \ref{Lem:GinvTrace}, we can define an antilinear operator $J_0$ on $H$ by $J_0\big( [a] \big) = [a^*]$. This operator preserves the norm on $H$:
$$ 
\langle J_0\big( [a] \big), J_0\big( [b] \big) \rangle = \langle [a^*], [b^*] \rangle = \tau( a b^*) = \tau(b^* a) = \langle [b], [a] \rangle 
$$
since $\tau$ is a trace. It thus extends to an antilinear operator on all of $\Hh_0$. We then set $J := J_0 \otimes J_S$ -- which is well-defined since both $J$ and $J_S$ are antilinear. Moreover, both $J_0$ and $J_S$ are norm-preserving which implies that $J$ is also norm-preserving. 

As both $a$ and $J b^* J^{-1}$ are bounded (linear) operators on $\Hh$, it suffices to check the commutation condition on the subset $H \otimes S \subseteq \Hh_0 \otimes S$, \textsl{i.e.} 
\begin{multline*}
 a \, J b^* J^{-1} ([c] \otimes s )= a J ([b^* c^*] \otimes J_S^{-1} s) = [a c b] \otimes s \\
= J ([b^* c^* a^*] \otimes J_S^{-1} s) =J b^* J ([a c] \otimes s) = J b^* J \, a ( [c]  \otimes s).
\end{multline*}

\medbreak

Regarding the first order condition, since both $[ D, a]$ and $J b^* J^{-1}$ are bounded linear operators, it suffices to prove the commutation on $[c] \otimes s$ for $c \in A$ and $s \in S$. We then have:
\begin{multline*}
[ D, a] \, J b^* J^{-1} \big([c] \otimes s \big) =  \left( \sum \partial_j^\Aa(a) \otimes F_j \right)  \big( [ c b] \otimes s \big) = \sum [ \partial_j^\Aa(a) c b] \otimes F_j s\\
 = J b^* J^{-1}  \left( \sum [ \partial_j^\Aa(a) c] \otimes F_j s \right) = J b^* J^{-1} \, [ D, a]\big([c] \otimes s \big).
\end{multline*}
Finally, the existence of a selfadjoint extension $\widetilde{D}$ of $D$ follows immediately from the real structure $J$ and Lemma \ref{Lem:UnboundExt}.
\end{proof}

\section{Covariant representation of compact Lie group}
\label{Sec:Compact}

Going back to general covariant representations, in the case of a \emph{compact} Lie group, we do not need to choose a selfadjoint extension of $D$: the operator is essentially selfadjoint.
\begin{Prop}
\label{Prop:CovRep}
Let $\Hh_0$ be a Hilbert space endowed with a covariant representation of $(A, G)$. If $G$ is \emph{compact}, then the operator $D$ of Prop. \ref{Prop:BasicDirac} is essentially selfadjoint.
\end{Prop}

\begin{Rem}
In full generality, if in Prop. \ref{Prop:BasicDirac} the operator $D$ is essentially selfadjoint, then Properties $(i)$--$(ii)$ of Prop. \ref{Prop:BasicDirac} and $(i)$--$(ii)$ of Prop. \ref{Prop:BasicDirac2} (if applicable) also hold for $\overline{D}$ -- see for instance Prop. 2 in the appendix of \cite{TrSpPiCr-Paterson}.
%
%
\end{Rem}

\begin{proof}
Peter-Weyl's decomposition theorem enables us to write $\Hh_0$ as a Hilbertian sum of $G$-representations:
\begin{equation}
\label{Eqn:DecompHh0}
\Hh_{0} = \bigoplus_{\ell} E_{\ell} \otimes \C^{m_\ell}
\end{equation}
where 
\begin{itemize}
\item
$\ell$ is a multi-index labelling the \emph{highest weight} of a representation of $G$,
\item
$E_\ell$ is a finite dimensional Hilbert space, endowed with the representation $\pi_\ell$ of $G$ of highest weight $\ell$,
\item
$m_{\ell}$ is the multiplicity of $E_\ell$ inside $\Hh_0$.
\end{itemize}
For each $\ell$, we can fix $m_\ell$ subspaces $E_{\ell, k} \subseteq \Hh_0$ for $k = 1, \ldots , m_\ell$ which are pairwise orthogonal, unitarily equivalent to $E_\ell$ and exhaust the $E_\ell$ component of $\Hh_0$. Denoting $P_{\ell, k}$ the associated projections on $\Hh_0$ and $Q_{\ell, k} := P_{\ell, k} \otimes 1_S$, we have $Q_{\ell, k} \dom(D) \subseteq E_{\ell, k} \otimes S \subseteq \dom(D)$ 
and $Q_{\ell, k}$ commutes with $D$. To prove that $D$ is essentially selfadjoint, it suffices to prove that $\Ran(D +i)$ and $\Ran(D -i)$ are dense (see \cite{ReedSimonI}, Corollary p.257). 

Using the decomposition \eqref{Eqn:DecompHh0} and the commutation of $Q_{\ell, k}$ with $D$, it suffices to prove that for each $\ell, k$, $E_{\ell, k} \otimes S = \Ran(Q_{\ell, k} D \pm i)$. Since $Q_{\ell, k} D $ is a symmetric operator on the finite dimensional space $E_{\ell, k} \otimes S$, it induces a basis of eigenvectors whose eigenvalues are real. This implies that both $Q_{\ell, k} D + i$ and $Q_{\ell, k} D - i$ are surjective and completes the proof.
\end{proof}

\section{Ergodic actions}
\label{Sec:Main}

In the particular case of \emph{ergodic actions} of {compact} Lie group, we can even estimate the summability of the closure $\overline{D}$ of $D$. To lighten notation, we sometimes write $D$ instead of $\overline{D}$.
\begin{Def}
The action $\alpha$ of $G$ on $A$ is \emph{ergodic} if the $G$-invariants elements are reduced to $\C$, \textsl{i.e.} if $\forall g \in G, \alpha_g(a) = a$, then $a \in \C 1_A$.
\end{Def}

\begin{Rem}
\label{Rem:Trace}
If $G$ is compact and the action is ergodic, then H\o{}egh-Krohn--Landstad--St\o{}rmer theorem (see Theorem 4.1 p.82 of \cite{ErgodCpctGpHKLS}) proves that the unique $G$-invariant state is actually a trace $\tau$. Hence the existence of a $G$-invariant trace $\tau$ is automatic! 
\end{Rem}

We now need a brief reminder regarding \emph{Dixmier trace ideals}, for which we follow Chapter IV of \cite{NCG}. 
More information on symmetrically normed operator ideals is available in \cite{Simontrace}. 

\begin{Def}
\label{Def:DixmierSum}
The ideal $\Ll^{p^+}$ (also denoted $\Ll^{(p,\infty )}$ in \cite{NCG}) is the set of all compact operators $T$ on $\Hh$ s.t. 
$$ 
\sup_k \frac{ \sigma_k(T)}{k^{(n-1)/n}} <  \infty 
$$
where $\sigma_k$ is defined as the supremum of the trace norms of $T E$, when $E$ is an orthonormal projection of dimension $k$, \textsl{i.e.}
$$
 \sigma_k(T) := \sup \{ \| T E \|_1 , \dim E = k \}.
$$
Equivalently, $\sigma_k(T)$ is the sum of the $k$ largest eigenvalues (counted with their multiplicities) of the positive compact operator $|T| := T^* T$. The elements of $\Ll^{p^+}$ are called \emph{$p^+$-summable} (or $(p, \infty )$-summable -- see \cite{NCG}, Sect. IV.2 $\alpha$ p.299 and following).

A spectral triple is \emph{$p^+$-summable} if $(1 + D^2)^{-1/2} \in \Ll^{p^+}$ (compare \cite{EltNCG}, Def. 10.8 p.450).
\end{Def}
%

Examples of $n^+$-summable spectral triples include spin manifolds of dimension $n$ equipped with their Dirac operators -- see \cite{EltNCG}, Theorem 11.1 p.488 and Theorem 7.18 p.293. This last property is related to Weyl's theorem. More material on this topic is available in \cite{EllipticRoe}, especially Chapter 8 therein. We are now ready to state properly:
\begin{Th}
\label{Thm:Main}
If $G$ is a \emph{compact} Lie group of dimension $n$ acting \emph{ergodically} on the unital $C^*$-algebra $A$,
\begin{enumerate}[1.]
\item
using the trace $\tau$ of Rem. \ref{Rem:Trace}, $\Hh_0 := \GNS(A, \tau)$ is endowed with a covariant representation of $(A,G)$ (see Prop. \ref{Prop:CovRep}) and the expression
\begin{equation}
\label{Eqn:Dirac2}
D = \sum_{j=1}^n \partial _j \otimes F_j ,
\end{equation}
(see Prop. \ref{Prop:BasicDirac}) defines on $\dom(D) = \Hh_0^\infty \otimes S$ an essentially selfadjoint operator $D$ s.t. $(A, \Hh_0 \otimes S, \overline{D})$ is a $n^+$-summable spectral triple on $A$, which is even when $n$ is even, has a real structure and satisfies the first order condition;
\item 
given a covariant representation of $(A,G)$ on $\Hh_0$ s.t. $\Hh_0^\infty $ is a finitely generated projective module on $\Aa$, the spectral triple $(A, \Hh_0 \otimes S, \overline{D})$ obtained from \eqref{Eqn:Dirac2} is $n^+$-summable and even if $n$ is even.
\end{enumerate}
\end{Th}

\begin{Rem}
This Theorem finally proves Point (i) of Def. \ref{Def:TrSp}! The other properties of the spectral triple are immediate consequences of Prop. \ref{Prop:BasicDirac} and \ref{Prop:BasicDirac2}.

\smallbreak

In point 2., the notation $\Hh_0^\infty $ is the same as in \eqref{Def:Hh0inf}. The condition on $\Hh_0^\infty $ as a $\Aa$-module mimicks the \emph{finiteness} Axiom (5) found in \cite{Grav} p.160.
\end{Rem}

\begin{Rem}
Under the above hypotheses, $D$ actually does not depend on the choice of orthonormal basis $(X_i)$ (see Not. \ref{Not:Base}), but the proof of this fact would require a more careful treatment of functoriality, \textsl{i.e.} using $\C l(\AlgLie ')$ instead of $\C l(n)$.
\end{Rem}

\begin{Rem}
The crucial point of the proof below is to control the multiplicities appearing in Peter-Weyl's decomposition \eqref{Eqn:DecompHh0}. Here, we rely on ergodicity and an estimate provided by \cite{ErgodCpctGpHKLS}. However, other means of controlling these multiplicities should lead to analogs of Theorem \ref{Thm:Main} for more general settings.
\end{Rem}

\begin{proof}

We begin with point 1. We are going to prove that $D$ has compact resolvent and is finitely summable by comparing it to the operator $D_{\text{ref}}$ defined by \eqref{Eqn:Dirac} acting on the Hilbert space $\Hh_{\text{ref}} :=L^2(G) \otimes S$ -- equipped with the left regular representation of $G$ on $L^2(G)$.

\medbreak

Since the tangent space $T G$ is trivial, $D_{\text{ref}}$ is actually a Dirac operator on the compact manifold $G$ -- equipped with its trivial spin structure. In particular, $D_{\text{ref}}$ has compact resolvent and is $n^+$-summable.

We can now decompose $\Hh_{\text{ref}}$ using Peter-Weyl's theorem. The result is a Hilbertian direct sum:
\begin{equation}
\label{Eqn:DecompL2G}
\Hh_{\text{ref}} = \bigoplus_{\ell} E_{\ell} \otimes \C^{d_\ell} \otimes S
\end{equation}
where 
\begin{itemize}
\item
$\ell$ is a multi-index labelling the \emph{highest weight} of a representation of $G$,
\item
$E_\ell$ is a finite dimensional Hilbert space, endowed with the representation $\pi_\ell$ of $G$ of highest weight $\ell$,
\item
$d_{\ell}$ is the dimension of $E_\ell$, which is also the multiplicity of $E_\ell$ inside $L^2(G)$.
\end{itemize}
We want to perform the same decomposition for $\Hh_0$. We first decompose $A$ into its \emph{spectral subspaces} (also called \emph{isotypical components}): given a highest weight $\ell$, we interpret the associated representation $\pi_\ell$ on $E_\ell$ as a $d_\ell \times d_\ell$-matrix and write $\chi_\ell(g) = d_\ell \Tr( \pi_\ell(g^{-1}) )$ for its normalised character. The associated spectral subspace $A_\ell$ is the norm closed subspace defined by:
$$ 
A_\ell := \overline{ \left\{ \int_G \chi_\ell(g) \alpha_g(a) dg \middle| a \in A \right\}} \subseteq A. 
$$
Relation (2.2.40) p.45 of \cite{DerivBratteli} proves that the (algebraic) direct sum $A^{\text{alg}} :=\bigoplus^{\text{alg}} A_\ell$ is norm dense in $A$. From \cite{ErgodCpctGpHKLS} p.76, 
we get that $A_\ell$ decomposes into a direct sum of irreducible components, each equivalent to $E_\ell$. Moreover, Prop. 2.1 in the same article ensures that the multiplicity $m_\ell$ of $E_\ell$ inside $A_\ell$ satisfies $m_\ell \leqslant d_\ell$ -- thereby proving that the dimension of $A_\ell$ is bounded by $d_\ell^2$.

\medbreak

Relying on the dense subset $A^{\text{alg}}$ of $A$, it is easy to prove that the unique $G$-invariant $\tau$ is faithful and therefore $\Hh_0$ is obtained as Hilbertian sum:
$$\Hh_0 = \bigoplus E_\ell \otimes \C^{m_\ell} .$$
Comparing the above expression with \eqref{Eqn:DecompL2G}, it is clear that there is an inclusion $\iota_0 \colon  \Hh_0 \to L^2(G)$ which commutes with the action of $G$. This inclusion extends to an inclusion $\iota \colon \Hh \to \Hh_{\text{ref}}$.

\medbreak

Since $D_{\text{ref}}$ has compact resolvent, $\Hh_{\text{ref}}$ admits a Hilbertian basis of eigenvectors. If for each $\ell$, we choose a decomposition $\bigoplus_{k=1}^{d_\ell} E_{\ell, k}$ of the term $E_\ell \otimes C^{d_\ell}$ in \eqref{Eqn:DecompL2G}, it is easily checked that the associated projections $P_{\ell, k}$ intertwin the action of $G$ on $\Hh_{\text{ref}}$ and thus commutes with $D_{\text{ref}}$. Moreover, we can pick the spaces $E_{\ell, k}$ so that if $k \leqslant m_\ell$, then $E_{\ell, k} \in \iota(\Hh)$ and if $k > m_\ell$, then $E_{\ell, k}$ is orthogonal to $\iota(\Hh)$.

\medbreak

Hence, we can choose a basis of $\Hh_{\text{ref}}$ made of eigenvectors for $D_{\text{ref}}$. Since $P_{\ell, k}$ commutes with $D_{\text{ref}}$, we can choose a basis which is compatible with the decomposition $1_{\Hh_{\text{ref}}} = \bigoplus P_{\ell, k}$. It is then clear that $D_{\text{ref}}$ and $D$ coincide on the relevant blocks $E_{\ell, k}$, and we therefore get a basis of $\Hh$ made of eigenvectors for $D$. The associated eigenvalues are the same for $D_{\text{ref}}$ and $D$. In particular, since $D_{\text{ref}}$ has compact resolvent, so does $D$.

\bigbreak

To prove that $D$ is $n^+$-summable, notice first that a consequence of the previous discussion is that the eigenvalues $(\mu_k)_{k \in \N}$ of $|D|$ are the same as those $( \lambda_k )_{k \in \N}$ of $|D_{\text{ref}}|$, but they have lower (possibly vanishing) multiplicities. This implies that if the eigenvalues $(\mu_k)$ and $(\lambda_k)$ (repeated with their multiplicities) are in increasing order, then for all $k \in \N$, $\lambda_k \leqslant \mu_k$. Indeed, given an increasing sequence, suppressing some terms leads to a sequence which increases ``faster''.

\smallbreak

Formally, we want to prove that $(1 + D^2)^{-1/2}$ is in the ideal $\Ll^{n^+}$ 
and we already know that $(1 + D_{\text{ref}}^2)^{-1/2} \in \Ll^{n^+}$. 

Since the function $f(x) =(1+ x^2)^{-1/2}$ is decreasing on $\R^+$, we see that if $(\lambda_k)$ is the sequence of eigenvalues of $|D_{\text{ref}}|$ in \emph{increasing} order, then $(\lambda'_n) :=\big( f(\lambda_k)\big)$ is the sequence of eigenvalues of $(1 + D_{\text{ref}}^2)^{-1/2}$ in \emph{decreasing} order. Moreover, as $\lambda_k \leqslant \mu_k$, we have $\lambda_k' = f(\lambda_k) \geqslant \mu'_k := f(\mu_k)$. In the notations of Def \ref{Def:DixmierSum}, we have:
\begin{align*}
\sigma_{k, \text{ref}} := \sigma_k \left( (1 +D_\text{ref}^2)^{-1/2} \right) &= \sum_{p = 0}^{k-1} \lambda_p'
&
\sigma_k := \sigma_k \left( (1 +D^2)^{-1/2} \right) &= \sum_{p = 0}^{k-1} \mu_p'
\end{align*}
and the monotony of $f$ implies $ \sigma_{k, \text{ref}} \geqslant \sigma_k$. That $(1 + D_{\text{ref}}^2)^{-1/2} \in \Ll^{n^+}$ means  
$$ 
\left\| \left( 1+D_{\text{ref}}^2 \right)^{-1/2} \right\|_{n^+} = \sup_k \frac{ \sigma_{k, \text{ref}}}{k^{(n-1)/n}} <  \infty.
$$
This in turn implies
$$ \left\| \left( 1+D^2 \right)^{-1/2} \right\|_{n^+} = \sup_k \frac{ \sigma_k}{k^{(n-1)/n}} \leqslant  \sup_k \frac{ \sigma_{k, \text{ref}}}{k^{(n-1)/n}} < \infty $$
and completes the proof that $D$ is $n^+$-summable.

\bigbreak

We now turn our attention to point 2. The proof is essentially the same, but we need to compare $D$ with $D_\text{ref,$k$} := D_\text{ref} \otimes 1_k$ acting on $L^2(G) \otimes S \otimes \C^k$ instead of $D_\text{ref}$. The finiteness condition on $\Hh_0^\infty $ can be written $\Hh_0^\infty \simeq p \Aa^k$, thus yielding a covariant inclusion $\Hh_0 \subseteq \GNS( A \otimes \C^k, \underline{\tau})$ where $\underline{\tau}$ is defined on $A \otimes \C^k$ by:
$$ \underline{\tau}(a_1, \ldots , a_k) = \frac{1}{k}\sum_{j=1}^k \tau(a_j) .$$
This is indeed a state since it is the restriction to diagonal matrices of the state $\tau \otimes \Tr$ defined on $A \otimes M_k(\C)$.

$D_\text{ref,$k$}$ is $n^+$-summable since functional calculus leads to
$$ ( 1+ D_\text{ref,$k$}^2)^{-1/2} = (1 + D_\text{ref}^2)^{-1/2} \otimes 1_k ,$$
and this is a finite sum of operators $\sum_j (1+D_\text{ref}^2)^{-1/2} \otimes e_{jj}$ which are all in the ideal $\Ll^{n^+}$. Another way to prove this property would be to use the ``scale invariance'' considered in \cite{NCG} IV.2.$\beta$ p.305.

\smallbreak

In any case, the argument concludes as for point 1.: using a Peter-Weyl decomposition, we see that $D_\text{ref,$k$}$ and $D$ coincide on the relevant irreducible components and a comparison of multiplicities leads to the $n^+$-summability of $D$.
\end{proof}

\section{Remarks, examples and counterexamples}
\label{Sec:Example}

We begin this section with a few comments about Theorem \ref{Thm:Main}: 
\begin{itemize}
\item
Our main result only gives an upper bound on summability. In general we cannot do better: if $G$ acts ergodically on $A$, then letting $K$ act trivially on $A$, we obtain an ergodic $G \times K$-action.
\item
For $n^+$-summable operator $D$, we know that a Hochschild cocycle $\varphi$ on $\Aa$ can be defined by:
\begin{equation}
\label{Eqn:Cocycle}
\varphi(a^0, a^1, a^2) = \lambda_n \Tr_\omega(\gamma a^0 [D, a^1] \cdots [D, a^n]),
\end{equation}
as stated in \cite{NCG}, IV.2 Theorem 8 p.308 -- which was later improved in \cite{localtrace95}. In the case of noncommutative tori (see Example \ref{Ex:ToreNC} below), the equation \eqref{Eqn:Cocycle} leads (up to renormalisation) to the cyclic cocycle:
$$
\varphi(a^0, a^1, a^2) = \tau\Big( a^0 (\partial _1(a^1) \partial_2(a^2) - \partial _2(a^1) \partial _1(a^2)) \Big).
$$
Would it be possible in the setting of Theorem \ref{Thm:Main} (1) to relate \eqref{Eqn:Cocycle} to cyclic cocycles as constructed in \cite{NCG}, III.6 example 12 c) p.256?
\item
Another interesting improvement would be to find sufficient conditions for the Poincaré duality (Axiom (6') of \cite{Grav}) to hold. However, this property would really depend on the algebraic structure on $A$, and not just on the multiplicity of its spectral subspaces.
\end{itemize}

To illustrate our results and hypotheses we give examples showing that frequently spectral triples arise from Lie group actions, and we hint at some possible interesting generalizations:

\begin{Ex}
\label{Ex:ToreNC}
The spectral triples on noncommutative tori were among the first examples considered by Connes. Indeed, they already appear in his article \cite{Grav} (see p.166) in which he defines the notion. His original example was in dimension $2$ but the construction was later extended	 to include noncommutative tori of any dimension (see \cite{EltNCG}, section 12.3 and especially p.545). 
We illustrate the notion in the two-dimensional case which was studied as early as in \cite{RieffTor}: the $C^*$-noncommutative torus $A_\theta$ is the universal $C^*$-algebra generated by two unitaries $U$ and $V$ subject to the relation $UV=e^{2\pi i\theta}VU$ for $\theta \in \R$.

It is well known that the actions of $T^n$ on the noncommutative tori (even of dimension $n$) are ergodic (see \cite{EltNCG} p.537). Hence our construction fully applies to these algebras. The unbounded operator $D$ obtained from our Theorem \ref{Thm:Main} is exactly the same as the operator defined in \cite{EltNCG} (12.24) p.545. The $n^+$-summability of $D$ corresponds to Prop. 12.14 p.545  
of \cite{EltNCG} and is a sharp estimate of summability in this case. 
\end{Ex}

\begin{Ex} 
\label{Ex:Tore1NC}
It is easy to see that some hypothesis on the fixed-point algebra is unavoidable. For example, already the action of $S^1$ on $\mT^2$ by $\lambda.(z_1,z_2):=(\lambda z_1,z_2)$ yields an unbounded operator on the Hilbert space $L^2(\mT^2)$ which certainly does not have compact resolvent because it has an infinite-dimensional kernel. 

However, note that this operator has ``compact (even summable) resolvent with coefficients in $C(\mT^2)$'', \textsl{i.e.} $(1+D^2)^{-1}\otimes 1\in \mc{L}^p\otimes_\pi C(\mT)$. In this sense, it yields an (unbounded) $\mathcal{LC}$-Kasparov module from $C(\mT^2)$ to $\mathcal{L}^p\otimes_{\pi}C(\mT)$ in the sense of \cite{MR2964680}.
\end{Ex}

\begin{Ex}
Quantum Heisenberg Manifolds (QHM) provide another illustration of our results. This family of algebras was introduced in \cite{RieffelDefQuant} as an example of Rieffel's strict quantization-deformation and largely studied since, \textsl{e.g.} \cite{AbadieEE, AbadieFixedPts, PairingsQHM, ChernGrensingGabriel, GeomQHM} to name just a few elements of the available litterature. In particular, QHM admit an ergodic action of the (non-compact) Heisenberg group $G$ and a unique $G$-invariant trace $\tau$ -- yielding properties very similar to those assumed in Theorem \ref{Thm:Main}. In the article \cite{GeomQHM}, Chakraborty and Sinha constructed a family of spectral triples on QHM whose Hilbert space is obtained from $\tau$ by the GNS construction. The expression of the operator is given by \eqref{Eqn:Dirac} (compare Prop. 9 of \cite{GeomQHM}) and they went further, proving summability in Theorem 10 p.431. 

When applied to Quantum Heisenberg manifolds, our results recover a real structure and thus complement \cite{GeomQHM}. However, the non-compactness of the Heisenberg group prevents us from reproducing summability. We will encounter a similar phenomenon in Example \ref{diracdual} below.  We thus plan to improve our construction in another article \cite{TrSpPCGGG}, elaborating on the present one.
\end{Ex}

\begin{Ex}
Another, related, example is given by Kasparov's Dirac element (see for example \cite{MR554420} or \cite{MR1388299} for the Dirac-dual Dirac method at work). It is easy to generalize our results in the following two ways: first we can treat more general modules than the uniquely determined spinor module we used in our construction of the Dirac-type operator. In fact, it suffices that the module $S$ used in the representation be a complex module over $\C l(n)$ (which carries a real structure if one wants to recover Prop. \ref{Prop:BasicDirac}). In the ``Real'' case, this means that the module is equipped with a ``Real'' structure which is compatible with the canonical ``Real'' structure on $\mathbb{C}l(n)$. Secondly, one may include a ``Real'' structure on the $C^*$-algebra $A$ which is preserved by the action of $G$. ``Real structures'' in this sense were already used by Kasparov in his very first definition of $KK$-theory (see, for example, \cite{MR1267059} for an overview). 

In order to carry over the summability results for the associated spectral triple it suffices to decompose the ``spinor module'' into irreducible representations.

With this slight generalization, it is possible to include Kasparov's Dirac element into our framework. It is given by the Hilbert space $\Omega_\C(\R^n)$ of $L^2$-forms on $\R^n$ on which $C(\R^n)$ acts naturally, together with the Hodge-Dirac operator. Our Proposition \ref{Prop:BasicDirac2} is in the spirit of Wolf's theorem  (\cite{LawsonMichelsohn} Theorem 5.7). However, the operator one obtains has continuous spectrum and therefore does not have compact resolvent as in the example of the Heisenberg manifold. This shows that some hypothesis is necessary on the relative size of the Lie group compared with the algebra in order to apply our techniques. Note however that for every function $f\in C(\R^n)$ with compact support the operator $f(i+D)^{-1}$ is in fact compact and the mapping $f\mapsto f(i+D)^{-1}$ is continuous (this also holds for Schwartz functions). In this sense, it should be possible to generalize our results to noncompact Lie groups by passing to nonunital spectral triples.
\end{Ex}

\begin{Ex}
\label{diracdual}
A less simple-minded example is provided by the harmonic oscillator $d+d^*+c(x)$ acting as in the last example on $\Omega_\C(\R^n)$, where $c(x)$ denotes Clifford multiplication by $x$ and $d +d^*$ the Hodge-Dirac operator. This operator has recently come to our attention in \cite{MR2732067} and was also used heavily in \cite{MR2964680}. And it again has a group-theoretic interpretation: it can be obtained as the operator associated to an action of the Heisenberg group of $\R^n$ on $\R^n$. The fact that it has summable resolvent (which is classical and seen by using Hermite polynomials) shows that when the Lie group is non-compact, it may induce a selfadjoint operator in the above way which has nevertheless summable resolvent.

This example is closely related to the spectral triple on $C(\mT)$ obtained by our techniques: the class in the spectral triple on $C(\R)$ defined by our techniques is in a sense a unitalization of the one defined in this example.
\end{Ex}

\begin{Ex} Further examples can be obtained from \emph{(Cuntz-)Pimsner algebras} as defined in \cite{PiCrG}. Indeed, consider a $C^*$-correspondence $E$ over a $C^*$-algebra $A$, \textsl{i.e.} a (right) Hilbert module $E$ with a (faithful) left $A$-action $\phi \colon A \to \Ll(E)$. Such objects are called ``Hilbert bimodules'' in \cite{PiCrG}. These data $E$ and $A$ define a Pimsner algebra $\Oo_E$.

Assume moreover that a Lie group $G$ acts on $A$ and that there is a compatible $G$-action on $E$, in the sense of \cite{PiCrG}, Remark 4.10-(2). The latter then provides us with a $G$-action on the Pimsner algebra, which commutes with the canonical gauge action, thereby defining an action of $G\times S^1$ on $\mc{O}_E$ -- to which our theory may apply. However, even if $G$ acts ergodically on $A$, the action of $G \times S^1$ on $\Oo_E$ need not be ergodic, as illustrated by the Cuntz algebra $\Oo_2$ generated by $A = \C$ (on which the trivial group $G= \{1\}$ acts ergodically) and $E = \C^2$: the $S^1$-action induced on $\Oo_2$ is clearly not ergodic!

\medbreak

Yet, if we consider \emph{Hilbert bimodules} ($E$ with a left- \emph{and} a right-Hilbert module structure) instead of $C^*$-correspondences, the resulting Pimsner algebras are in fact \emph{generalised crossed products} (see \cite{AbadieEE}) and better results are available. For instance, if $G$ acts ergodically on $A$, then the induced $G \times S^1$-action on $\Oo_E$ is ergodic, and our theory applies to its fullest extent (namely Theorem \ref{Thm:Main}). This sort of results form the core of our forthcoming article \cite{TrSpPCGGG}.
\end{Ex}

\paragraph{Acknowledgement}

Both authors are grateful to D. Bahns and G. {Skandalis}.

\medbreak

Olivier \textsc{Gabriel}, \texttt{ogabriel@uni-math.gwdg.de}, 

Mathematisches Institut -- Universität Göttingen

Bunsenstr. 3-5 D--37 073 Göttingen, Germany

\bigbreak

Martin \textsc{Grensing}, \texttt{grensing@gmx.net}

D\'epartment de Math\'ematiques -- Universit\'e d'Orl\'eans, 

B.P. 6759 -- 45 067 Orl\'eans cedex 2, France.

\bibliographystyle{alpha} 
\bibliography{biblio}

\end{document}